\documentclass[10pt,reqno, amssymb, amscd]{amsart}

\usepackage[b5paper,top=1.2in,left=0.9in]{geometry}
\usepackage{verbatim}
\usepackage{amssymb}
\usepackage{amsmath}
\usepackage{graphicx}
\usepackage{appendix}
\usepackage{color}
\usepackage{amsthm}

\renewcommand{\d}{\mathrm{d}}
\newcommand{\D}{\mathrm{D}}

\newcommand{\e}{\mathrm{e}}

\newtheorem{Thm}{Theorem}[section]
\newtheorem{Lem}[Thm]{Lemma}
\newtheorem{Prop}[Thm]{Proposition}

\newtheorem{Rem}[Thm]{Remark}
\newtheorem{Def}[Thm]{Definition}
\newtheorem{Ex}[Thm]{Example}

\def\C{\mathbb{C}}
\def\Z{\mathbb{Z}}

\def\H{\mathbb{H}}

\def\fb{\mathfrak{b}}

\def\sl{\mathfrak{sl}}

\def\osp{\mathfrak{osp}}

\def\cE{\mathcal{E}}
\def\cF{\mathcal{F}}

\def\cH{\mathcal{H}}

\def\cK{\mathcal{K}}

\def\cR{\mathcal{R}}

\def\cU{\mathcal{U}}

\def\cW{\mathcal{W}}

\def\a{\alpha}
\def\b{\beta}
\def\c{\gamma}

\def\D{\Delta}

\def\d{\delta}
\def\e{\epsilon}

\def\h{\theta}

\def\l{\lambda}

\def\s{\sigma}

\def\bC{\textbf{C}}

\def\bE{\textbf{E}}

\def\bF{\textbf{F}}

\def\bL{\textbf{L}}

\def\bo{\textbf{o}}

\def\bQ{\textbf{Q}}

\def\bR{\textbf{R}}

\def\bT{\textbf{T}}

\def\=>{\Longrightarrow}

\def\to{\longrightarrow}

\def\ox{\otimes}
\def\o+{\oplus}
\def\bo+{\bigoplus}

\def\<{\langle}
\def\>{\rangle}
\def\({\left(}
\def\){\right)}

\def\oo{\infty}
\def\cong{\equiv}
\def\^{\wedge}
\def\+{\dagger}

\def\inv{^{-1}}

\def\no{\nonumber}

\def\tab{\;\;\;\;\;\;}

\newcommand{\til}[1]{\widetilde{#1}}
\newcommand{\what}[1]{\widehat{#1}}

\newcommand{\veca}[1]{\begin{pmatrix}#1 \end{pmatrix}}

\newcommand{\case}[2][cccccccccccccccccccccccccccccccccccccccccc]{\left\{\begin{array}{#1}#2 \end{array}\right.}
\newcommand{\Eq}[1]{\begin{align}#1\end{align}}
\newcommand{\Eqn}[1]{\begin{align*}#1\end{align*}}

\begin{document}
\title{Q-operator and fusion relations for $U_q(C^{(2)}(2))$}
\author{Ivan Chi-Ho Ip}
\author{Anton M. Zeitlin}
\address{\newline
Ivan Chi-Ho Ip,\newline
Kavli Institute for the Physics and \newline
Mathematics of the Universe (WPI),\newline
University of Tokyo,\newline
Kashiwa, Chiba,\newline
277-8583 Japan\newline
ivan.ip@ipmu.jp\newline
http://member.ipmu.jp/ivan.ip}
\address{ 
\newline 
Anton M. Zeitlin,\newline
Department of Mathematics,\newline
Columbia University,\newline
Room 509, MC 4406,\newline
2990 Broadway,\newline
New York, NY 10027,\newline
Max Planck Institute for Mathematics,\newline
Vivatsgasse 7, Bonn, 53111, Germany,\newline
IPME RAS, V.O. Bolshoj pr., 61, 199178,\newline 
St. Petersburg\newline
zeitlin@math.columbia.edu,\newline
http://math.columbia.edu/$\sim$zeitlin \newline
http://www.ipme.ru/zam.html  }

\begin{abstract}
The construction of the Q-operator for twisted affine 
superalgebra $U_q(C^{(2)}(2))$ is given. It is shown that the corresponding prefundamental representations give rise to evaluation modules some of which do not have a classical limit, which nevertheless appear to be a necessary part of fusion relations.  
\end{abstract}

\maketitle



\section{Introduction}
The Q-operator and its generalizations are important ingredients in the study of quantum integrable models. Namely, eigenvalues of the transfer-matrices,  corresponding to various representations can be expressed in terms of eigenvalues of the Q-operator, which has less complicated analytic properties.  
These features of the Q-operators were first noticed by  
Baxter the early 70s in the case of vertex models. Later, after the quantum group interpretation of the quantum integrable models it was realized that the original Baxter Q-operator correspond to the integrable model based on the  simplest nontrivial quantum affine algebra $U_q(\what{\sl}(2))$. A natural question was to generalize this notion to the higher rank and give a proper representation-theoretic meaning to these fundamental building blocks for transfer matrix eigenvalues. The first idea in that direction was given in the papers of V. Bazhanov, S. Lukyanov and A. Zamolodchikov \cite{blz1}, \cite{blz2} in the context of the construction of integrable structure of conformal field theory: the interpretation of Q-operators for $U_q(\what{\sl}(2))$ as transfer-matrices 
for certain infinite-dimensional representations of the Borel subalgebra of $U_q(\what{\sl}(2))$.  Later their results were  generalized 
in \cite{bhk}, \cite{kojima} to the case of $U_q(\what{\sl}(n))$. Finally,  
in the recent preprint of E. Frenkel and D. Hernandez  \cite{hf} the full representation-theoretic description of Q-operators was given for large class of integrable models based on any untwisted quantum affine algebra $\cU_q(\mathfrak{g})$ and connected to the 
earlier description of the transfer-matrix eigenvalues via the q-characters \cite{fr}. The infinite-dimensional representations corresponding to the Q-operator, which the authors called "prefundamental representations" were constructed just before that in \cite{hj}.  

At the same time, some analogues of the Q-operators were constructed in this way in the case of superalgebras \cite{kz}, \cite{bt}, \cite{tsuboi}. In this article we improve the constructions of \cite{kz}. 
In that paper an attempt to construct the Q-operator and associated fusion relation for transfer matrices was made in the case of $U_q(C^{(2)}(2))\cong U_q(sl^{(2)}(2|1))$. However, the construction given there lead to only partial result: half of the resulting transfer matrices were built ``by hands" out of Q-operators and did not seem to correspond to any finite dimensional representation of $U_q(C^{(2)}(2))$. In this paper we solve this ambiguity, by allowing some representations to have no classical limit ($q\rightarrow 1$). The approach we are using allows to show explicitly the similarity between $U_q(A_1^{(1)})$ and $U_q(C^{(2)}(2))$ previously noticed on the level of universal $R$-matrices \cite{kt}. 

The structure of the article is as follows. In Section \ref{sec:osp21} we fix the notations 
and describe the relation between finite-dimensional representations of $U_q(\osp(2|1))$ and $U_q(\sl(2))$, previously noticed on the level of modular double \cite{iz}. The approach, which can be generalized to higher rank superalgebras is that we find representations of $U_q(\osp(2|1))$ inside the tensor product of finite-dimensional representation of $U_{-iq}(\sl(2))$ and two-dimensional Clifford algebra.  
Such representation splits into two irreducible representations which differ by the parity of the highest weight and have equal dimensions. It is notable that the even-dimensional irreducible representations obtained in this way 
do not have the classical limit. We also give explicit formulas for $R$-matrix in these representations.
In Section \ref{sec:evaluate} we consider evaluation modules for $U_q(C^{(2)}(2))$, which can be obtained in a similar fashion from evaluation modules of $U_{-iq}(A_1^{(1)})$. We explicitly find the resulting trigonometric $R$-matrix and its matrix coefficients (with the details of calculations in the Appendix). We also introduce in Section \ref{sec:evaluate} the prefundamental representations for $U_q(C^{(2)}(2))$ and study in detail the relations in the Grothendieck ring of prefundamental representations combined with evaluation representations. The relations in the Grothendieck ring lead to relations between transfer-matrices and Q-operators: in Section \ref{sec:transfer} we correct the constructions of \cite{kz}, where the integrable structure of superconformal field theory was studied, now changing ``fusion-like" relations by the true fusion relations.

\section{Quantum superalgebra $U_q(\osp(2|1))$ and its representations}\label{sec:osp21}

We define the quantum superalgebra $U_q(\osp(2|1))$ as follows. It is a Hopf algebra generated by even element $\cK$ and odd elements $\cE$ and $\cF$ such that
\Eqn{
\{\cE,\cF\}&:=\cE\cF+\cF\cE=\frac{\cK-\cK\inv}{q+q\inv},\\
\cK\cE&=q^2\cE\cK,\\
\cK\cF&=q^{-2}\cF\cK,
}
where the corresponding coproduct is:
\Eqn{
\D(\cE)&=\cE\ox \cK+1\ox \cE,\\
\D(\cF)&=\cF\ox 1+\cK\inv\ox \cF,\\
\D(\cK)&=\cK\ox \cK.
}

Let us choose the (odd) Clifford generators $\xi, \eta$ satisfying \Eq{\xi^2=\eta^2=1, \tab \xi\eta= -\eta\xi}
which acts in the space $\C^{1|1}:=span\{|+\>,|-\>\}$, where $|-\>$, $|+\>$ are odd and even vectors correspondingly, by 
\Eq{\pi(\xi)=\veca{0&1\\1&0}, \tab \pi(\eta)=\veca{0&i\\-i&0},} such that 

\Eq{\pi(i\xi\eta) = \veca{1&0\\0&-1}.}

The following notation will play a crucial role in relating the superalgebra and the classical case via the spinor representation:
\begin{Def}
We denote by 
\Eq{q_*:=-iq}
and writing $q:=e^{\pi ib^2}, q_*:=e^{\pi ib_*^2}$, we have
\Eq{b^2=b_*^2+\frac{1}{2}.}
\end{Def}

Then we have the following proposition observed in \cite{iz}, which can be proved by direct computation.
\begin{Prop}\label{supermap} If $E,F,K$ generate $U_{q_*}(\sl(2))$, then
\Eq{\cE=\xi E, \tab \cF=\eta F, \tab \cK=i\xi\eta K}
generate $U_q(\osp(2|1))$.
\end{Prop}

Therefore, we are now able to relate the representations of 
$U_{q_*}(\sl(2))$ and $U_q(\osp(2|1))$. Let us do it explicitly.

Consider the $s+1=2l+1$ dimensional representation $V_s$ of $U_{q_*}(\sl(2))$ with basis
$$e_m^{l},\tab m=-l,...,l$$
and action
\Eqn{K\cdot e_m^{l}&=q_*^{2m}e_m^{l},\\
H\cdot e_m^{l}&=(2m)e_m^{l},\\
E\cdot e_m^{l}&=[l-m]_{q_*}e_{m+1}^{l},\\
F\cdot e_m^{l}&=[l+m]_{q_*}e_{m-1}^{l},
}
where formally $K=q_*^H$ and $[n]_q:=\frac{q^n-q^{-n}}{q-q\inv}$ is the quantum number.

The generators $\cE, \cF, \cK$ naturally act on $V_s\ox \C^{1|1}$ by means of the $U_{q_*}(\sl(2))$ action, and decomposes as 
\Eq{W_s=V_s\ox \C^{1|1}=W_s^+ \o+ W_s^-,}
where $W_s^\pm$ has highest weight $w_s^\pm=e_l^l \ox |\pm\>$ and spanned by
\Eq{W_s^\pm=span\{ w_s^\pm, \cF\cdot w_s^\pm, \cF^2\cdot w_s^\pm,..., \cF^{s}\cdot w_s^\pm\}.}
Let $e_{m,\pm}^{l}:=e_m^l\ox |\pm\>$ be the natural basis of  $V_s\ox \C^{1|1}$. Note that $e_{m,-}^{l}$ is an odd vector while $e_{m,+}^{l}$ is even.
Then the action of $\cE, \cF, \cK$ can be written explicitly as follows:

\begin{Prop}
\Eqn{ \cK\cdot e_{m,\pm}^{l}&=\pm q_*^{2m}e_{m,\pm}^{l},\\
&=\pm i^{-2m}q^{2m}e_{m,\pm}^{l},\\
\cE\cdot e_{m,\pm}^{l}&=[l-m]_{q_*}e_{m+1,\mp}^{l},\\
&=i^{l-m-1}\{l-m\}_q e_{m+1,\mp}^{l},\\
\cF\cdot e_{m,\pm}^{l}&=\mp i[l+m]_{q_*}e_{m-1,\mp}^{l}\\
&=\mp i^{l+m}\{l+m\}_{q}e_{m-1,\mp}^{l}\\
}
where $\{n\}_{q}:= \frac{q^{-n}-(-1)^{n}q^{n}}{q+q\inv}=i^{1-n} [n]_{q_*}$.
\end{Prop}

We notice that the representations of even dimension is something which we do 
not encounter in the classical case, namely all the finite-dimensional irreducible representations of Lie superalgebra $\osp(2|1)$ are odd-dimensional.

\begin{Rem}
It is well known that the finite-dimensional irreducible representations of $\osp(2|1$) Lie superalgebra have odd dimension only (see e.g. \cite{dictionary}). One can relate  $W_s^{\pm}$ for even $s$ to those by considering classical limit. Due to our normalization, to do that one has to proceed through the following steps. First, one has to rescale $e^l_m$ so that $Fe^l_m=e^l_{m-1}$ and renormalize E so that $E'=\frac{q+q^{-1}}{q-q^{-1}}E$. Then $\mathcal{E}'=\xi E'$, and $\mathcal{F}$ are such that the commutation relations on $W_s^{\pm}$ in the limit $q\to 1$ are such that $[\mathcal{E}', \mathcal{F}]=H$, i.e. the commutation relations of $\osp(2|1)$. Such limiting procedure is not possible in the case of even-dimensional $W_s^{\pm}$ as the coefficients will not converge.
\end{Rem}

\begin{Ex} For $l=\frac{1}{2}$, the representation on $(W_1^\pm, \pi_1)$ with basis $\{e_{1/2,\pm}^{1/2},e_{-1/2,\mp}^{1/2}\}$ is given by
\Eqn{
\pi_1(\cK)&=\veca{\mp iq&0\\0&\mp iq\inv},\tab \pi_1(H)=\veca{1&0\\0&-1}, \\
\pi_1(\cE)&=\veca{0&1\\0&0},\tab \pi_1(\cF)=\veca{0&0\\\mp i&0}.
}
For $l=1$, the representation on $(W_2^\pm, \pi_2)$ with basis $\{e_{1,\pm}^{1},e_{0,\mp}^{1},e_{-1,\pm}^{1}\}$ is given by
\Eqn{
\pi_2(\cK)&=\veca{\mp q^2&0&0\\0&\mp1 &0\\0&0&\mp q^{-2}},\tab \pi_2(H)=\veca{2&0&0\\0&0&0\\0&0&-2}, \\
\pi_2(\cE)&=\veca{0&1&0\\0&0&i(q\inv -q)\\0&0&0},\tab \pi_2(\cF)=\veca{0&0&0\\\pm(q\inv -q)&0&0\\0&\pm i&0}.
}
\end{Ex}

Now we will find the formula for the $R$-matrix acting in tensor product of 
$W^{\pm}_s$.

Let \Eq{\exp_q(x)=\sum_{n=0}^\oo \frac{x^n}{\lceil n \rceil_q!}}
where $\lceil n \rceil_q=\frac{1-q^n}{1-q}$. The the following Theorem holds.
\begin{Thm}\label{R_osp}
The universal $R$ matrix is given by
$$\bR=Q \cR$$
where $Q:=C q_*^{\frac{H\ox H}{2}}$ with
$C:=\frac{1}{2}(1\ox 1+i\xi\eta\ox 1+1\ox i\xi\eta +\xi\eta\ox \xi\eta)$
such that 
\Eq{C\cdot |(-1)^{\e_1}\>\ox |(-1)^{\e_2}\> = (-1)^{\e_1\e_2} |(-1)^{\e_1}\>\ox |(-1)^{\e_2}\>,\tab \e_i\in\{0,1\},} and
\Eq{ \cR&:=\exp_{q_*^{-2}}(i(q_*\inv-q_*) \cE\ox \cF)\\
\no &=\exp_{-q^{-2}}(-(q+q\inv) \cE\ox \cF)\\
\no &=\sum a_n \cE^n \ox \cF^n}
where 
\Eq{a_n=(-1)^nq^{\frac{1}{2}n(n-1)}\frac{(q+q\inv)^n}{\{n\}_q!}}
\end{Thm}

The proof is given in Appendix.\\

Finally, let us give for completeness the explicit matrix coefficients of $R$.  
Namely, we find the pairing for $\bR_{l_1,l_2}=\bR|_{W^{\pm}_{s_1}\ox W^{\pm}_{s_2}}$
$$\left\<e_{m_1',\e_1}^{l_1}\ox e_{m_2',\e_2}^{l_2}, \bR_{l_1,l_2}(e_{m_1,\e_1'}^{l_1}\ox e_{m_2,\e_2'}^{l_2})\right\>$$
where $\epsilon_i\in\{0,1\}$ indicates the parity, namely $|\pm\>=|(-1)^{\epsilon}\>$.
Let us fix $l_1, l_2$ and write $e_{m,\e}^{l}$ for $e_{m,\pm}^{l}$.
\begin{Prop}
$$\left\<e_{m_1',\e_1}^{l_1}\ox e_{m_2',\e_2}^{l_2}, \bR_{l_1,l_2}(e_{m_1,\e_1'}^{l_1}\ox e_{m_2,\e_2'}^{l_2})\right\>=0$$
if $m_1'-m_1\neq m_2-m_2'$ or $m_1'-m_1=m_2-m_2'<0$. 

Otherwise let $n=m_1'-m_1$, we have
\Eqn{&\left\<e_{m_1',\e_1'}^{l_1}\ox e_{m_2',\e_2'}^{l_2}, \bR_{l_1,l_2}(e_{m_1,\e_1}^{l_1}\ox e_{m_2,\e_2}^{l_2})\right\>\\
=& i^{(l_1-m_1+l_2+m_2-1)n-2m_1'm_2'}(-1)^{\e_1\e_2+n}q^{\frac{1}{2}n(n-1)+2m_1'm_2'}\cdot\\
&\cdot\frac{(q+q\inv)^n}{\{n\}_q!}\frac{\{l_1-m_1\}_q!}{\{l_1-m_1-n\}_q!}\frac{\{l_2+m_2\}_q!}{\{l_2+m_2-n\}_q!}
}
In terms of $q_*$ and using the standard $[n]_{q_*}$ instead, we get
\Eqn{=&q_*^{\frac{1}{2}n(n-1)+2m_1'm_2'}(-1)^{\e_1\e_2}\frac{(q_*-q_*\inv)^n}{[ n ]_{q_*}!}\frac{[ l_1-m_1]_{q_*}!}{[ l_1-m_1-n]_{q_*}!}\frac{[ l_2+m_2]_{q_*}!}{[ l_2+m_2-n]_{q_*}!}}Note that there are no more $i$'s using the $q_*$ notation.
\end{Prop}
\begin{Ex}
For $W_1^+\ox W_1^+$, let the basis be $\{e_{1/2,+}^{1/2},e_{-1/2,-}^{1/2}\}\ox\{e_{1/2,+}^{1/2},e_{-1/2,-}^{1/2}\}$. Then $R$ is given by
$$\bR_{\frac{1}{2},\frac{1}{2}}=\veca{q_*^{\frac{1}{2}}&0&0&0\\0&q_*^{-\frac{1}{2}}&(1-q_*^{-2})q_*^{\frac{1}{2}}&0\\0&0&q_*^{-\frac{1}{2}}&0\\0&0&0&-q_*^{\frac{1}{2}}}$$
\end{Ex}
\begin{Ex}
For $W_2^+\ox W_2^+$, let the basis be $\{e_{1,+}^{1},e_{0,-}^{1},e_{-1,+}^{1}\}\ox \{e_{1,+}^{1},e_{0,-}^{1},e_{-1,+}^{1}\}$. Then $R$ is given by
$$\bR_{1,1}=\veca{
q_*^2&0&0&0&0&0&0&0&0\\
0&1&0&q_*^{2}-q_*^{-2}&0&0&0&0&0\\
0&0&q_*^{-2}&0&q_*^{-2}(q_*\inv-q_*)&0&(q_*^2-q_*^{-2})(1-q_*^{-2})&0&0\\
0&0&0&1&0&0&0&0&0\\
0&0&0&0&-1&0&(q_*^2-q_*^{-2})(q_*+q_*\inv)&0&0\\
0&0&0&0&0&1&0&q_*^{2}-q_*^{-2}&0\\
0&0&0&0&0&0&q_*^{-2}&0&0\\
0&0&0&0&0&0&0&1&0\\
0&0&0&0&0&0&0&0&q_*^2\\}
$$
\end{Ex}

Finally we give some remarks about the Casimir operator. In $U_{q_*}(\sl(2))$, it is known that the center is generated by the Casimir operator given by (up to some additive constant):

\Eq{\bC_{\sl(2)} = FE+\frac{q_*K+q_*\inv K\inv}{(q_*-q_*\inv)^2}}
By Proposition \ref{supermap}, it is obvious that $\bC_{\sl(2)}$ commutes with our generators. However, it is not an element of $U_q(\osp(2|1))$. Instead, the element
\Eq{\sqrt{\bC_{\osp(2|1)}} := \eta\xi \bC_{\sl(2)} = \cF\cE-\frac{q\cK-q\inv\cK\inv}{(q+q\inv)^2}} 
will be an element in $U_q(\osp(2|1))$ super-commuting with the generators $\{\cE, \cF, \cK\}$. 
By construction, it's square $\bC_{\osp(2|1)} := -\bC_{\sl(2)}^2$ is in the center of $U_q(\osp(2|1))$, given by
\Eq{\bC_{\osp(2|1)} =-\cF^2\cE^2 + \frac{q-q\inv}{(q+q\inv)^2}(q^2\cK+q^{-2}\cK\inv)\cF\cE + \frac{q^2\cK^2+q^{-2}\cK^{-2}}{(q+q\inv)^4}}
up to an additive constant. Under a rescaling of the generators, this is precisely the Casimir element of $U_q(\osp(2|1))$ found in \cite{kr}.

Now it is also clear that the representations $W_s^\pm$ correspond to the positive and negative spectrum of the square root of the Casimir element $\sqrt{\bC_{\osp(2|1)}} = \eta\xi \bC_{\sl(2)}$.

\section{Evaluation modules for $U_q(C^{(2)}(2))$ and prefundamental representations}\label{sec:evaluate}
The quantum affine superalgebra $U_q(C^{(2)}(2))$ is generated by $\cE_i,\cF_i,\cK_i$, $i=0,1$, where $\cE_i$ and $\cF_i$ are odd, with Cartan matrix given by $$A=\veca{2&-2\\-2&2}.$$

In particular, we have 
\Eq{\cK_i\cE_j=q^{a_{ij}}\cE_j\cK_i,\tab\cK_i\cF_j=q^{-a_{ij}}\cF_j\cK_i,}
and in addition the Serre relations
\Eq{\cE_i^3\cE_j+\{3\}_q \cE_i^2\cE_j\cE_i-\{3\}_q\cE_i\cE_j\cE_i^2-\cE_j\cE_i^3&=0\\
\cF_i^3\cF_j+\{3\}_q \cF_i^2\cF_j\cF_i-\{3\}_q\cF_i\cF_j\cF_i^2-\cF_j\cF_i^3&=0}
where $\{3\}_q=\frac{q^3+q^{-3}}{q+q\inv}$. Furthermore, for later convenience we modify the scaling of $\cF_i$ and use instead the following commutation relations:
\Eq{\{\cE_i,\cF_i\}=\frac{\cK_i-\cK_i\inv}{q+q\inv}.\label{rescale}}

\subsection{Evaluation modules for $U_q(C^{(2)}(2))$ and trigonometric $R$-matrix}

One check easily that we have the following spinor representation as in the $U_q(\osp(2|1))$ case:
\Eq{\cE_j=E_j\xi,\tab \cF_j=F_j\eta, \tab \cK_j=i\xi\eta K_j,}
and we also have the evaluation modules induced from $(A_1^{(1)})_{q_*}$ given by
\Eqn{E_1\mapsto \l E,&\tab E_0\mapsto \l F\\
F_1\mapsto \l\inv F,&\tab F_0\mapsto \l\inv E\\
K_1\mapsto K,&\tab K_0\mapsto K\inv
}

Then using the 2-dimensional representation of the Clifford algebra, we can consider its action as before on $V_s\ox \C^{1|1}$, and decompose it into $W_s(\l):=W_s(\l)^+\ox W_s(\l)^-$.

\begin{Prop}The action on the evaluation module $W_s(\l)^\pm$ with basis $e_{m,\pm}^l$, ${s=2l}$, $m=-l,...,l$, is given by
\Eqn{\cE_1\cdot e_{m,\pm}^l&=\l[l-m]_{q_*}e_{m+1,\mp}^l\\
\cE_0\cdot e_{m,\pm}^l&=\l [l+m]_{q_*}e_{m-1,\mp}^l\\
\cF_1\cdot e_{m,\pm}^l&=\mp i\l\inv [l+m]_{q_*}e_{m-1,\mp}^l\\
\cF_0\cdot e_{m,\pm}^l&=\mp i\l\inv[l-m]_{q_*}e_{m+1,\mp}^l\\
\cK_1\cdot e_{m,\pm}^l&=\pm q_*^{2m}e_{m,\pm}^l\\
\cK_0\cdot e_{m,\pm}^l&=\pm q_*^{-2m}e_{m,\pm}^l\\
\cK_\d\cdot  e_{m,\pm}^l&= e_{m,\pm}^l\\
}
\end{Prop}

In the case $s=1$, one can solve for the $R$ matrix explicitly.
\begin{Prop}
The $R$ matrix for $s=1$, $W_s(\l_1)^{\e_1}\ox W_s(\l_2)^{\e_2}$, $\e_i\in \{+,-\}$, is, up to scalar, given by

\Eq{\label{RC2} \bR\simeq\veca{1-z^2 q_*^2&0&0&0\\0&\e_1q_*(1-z^2)&\e_2z(1-q_*^2)&0\\0&\e_1z(1-q_*^2)&\e_2q_*(1-z^2)&0\\0&0&0&-\e_1\e_2(1-z^2q_*^2)},}
where $z=\frac{\l_2}{\l_1}$.
Alternatively, let $\l_i=e^{x_i}$, then we can cast it in trigonometric terms:
\Eq{\bR\simeq\veca{\sinh(x_1-x_2-\ln q_*)&0&0&0\\0&\e_1\sinh(x_1-x_2)&\e_2\sinh(\ln q_*)&0\\0&\e_1\sinh(\ln q_*)&\e_2\sinh(x_1-x_2)&0\\0&0&0&-\e_1\e_2\sinh(x_1-x_2-\ln q_*)}.}
\end{Prop}

In the general case, one has to calculate the action of the generators corresponding to the imaginary roots. The explicit calculation is given in the Appendix and the explicit form of the $R$-matrix is presented in Theorem \ref{RC2full}.

\subsection{Prefundamental representations and the Grothendieck ring}

Let us consider the Verma modules corresponding to evaluation modules of   $U_q(C^{(2)}(2))$. Namely, let us start from  the following representation of $U_q(\osp(2|1))$:
\Eq{\cW_s^\pm &= \{ \cF^k\cdot w_s^\pm\}_{k=0}^\oo,}
where $w_s^\pm:=e_l^l\ox |\pm\>$ as before such that $\cK\cdot w_s^\pm = \pm q_*^s w_s^\pm$.

Writing $|k\>_\pm:=\cF^k w_s^\pm$, the basis are related to $e_{m,\pm}^l$ of the $s+1$ dimensional module $W_s^\pm$ from before by
\Eq{\no |0\>_\pm&=w_s^\pm=e_{l,\pm}^l,\\
|k\>_\pm&=\cF^k \cdot w_s^\pm = \cF^k\cdot e_{l,\pm}^l=i^{-k}\frac{[2l]_{q_*}!}{[2l-k]_{q_*}!}e_{l-k,\pm(-1)^k}^l.}
Note that $|k\>_\pm$ is an even vector when $\pm(-1)^k = +1$.

This gives rise to the following evaluation module of the upper Borel part $\fb_+$ of  $U_q(C^{(2)}(2))$ on $\cW_s^\pm$:
\Eqn{\cE_0|k\>_\pm&=\l \cF|k\>_\pm = \l |k+1\>_\pm,\\
\cE_1|k\>_\pm&=\l \cE |k\>_\pm = \l[k]_{q_*}[s-k+1]_{q_*} |k-1\>_\pm,\\
\cK_0|k\>_\pm&=\cK\inv |k\>_\pm=\pm q^{2k}q_*^{-s}|k\>_\pm=\pm(-1)^k q_*^{2k-s}|k\>_\pm,\\
\cK_1|k\>_\pm&=\cK |k\>_\pm=\pm q^{-2k}q_*^{s}|k\>_\pm=\pm(-1)^k q_*^{-2k+s}|k\>_\pm.
}

Furthermore, we see that when $s\in \Z_{\geq 0}$, $\cW_s^\pm(\l)$ has a block diagonal form such that in the Grothendieck ring of the representation of $\fb_+$,
\Eq{[\cW_s^\pm(\l)]=[W_s^\pm(\l)]+[\cW_{-s-2}^{\pm(-1)^{s+1}}(\l)].}

Let us define the prefundamental (or $q$-oscillator) representation of $\fb_+$ of  $U_q(C^{(2)}_q(2))$. 
The $q$-oscillator algebra is generated by $\a_+,\a_-,\cH$ such that
\Eq{q \a_+\a_- + q\inv \a_-\a_+ = -\frac{1}{q+q\inv},\tab [\cH,\a_\pm]=\pm 2\a_\pm,}
where $\a_\pm$ are considered as odd elements.
We consider the Fock modules 
\Eq{\Pi_{\pm}=span\{\a_\pm^k|0\>_\pm :\cH|0\>_\pm = 0, \a_\mp|0\>_\pm =0\}_{k=0}^\oo,}
where the vacuum vectors $|0\>_\pm$ are even. Then we have an important Lemma.

\begin{Lem} The following substitution provides an infinite dimensional representation of $\fb_+$:
\Eq{\rho_\pm(\l): \cE_1=\l\a_\pm,\tab \cE_0=\l \a_\mp,\tab \cK_1=q^{\pm \cH},\tab \cK_0=q^{\mp\cH}.}
\end{Lem}

Let us consider the tensor product $\rho_+(\l\mu)\ox \rho_-(\l\mu\inv)$. The action of $\cE_i$ is given by
\Eqn{\cE_1&=\l(\mu \a_+\ox q^{-\cH}+1\ox \mu\inv \a_-)=:\l(a_-+b_-)\\
\cE_0&=\l(\mu \a_-\ox q^{\cH}+1\ox \mu\inv \a_+)=:\l(a_++b_+)}
so that we have the commutation relations
\Eqn{qa_-a_++q\inv a_+a_-&=-\frac{\mu^2}{q+q\inv},\\
qb_+b_-+q\inv b_-b_+&=-\frac{\mu^{-2}}{q+q\inv},\\
a_{\d_1}b_{\d_2}&=-q^{2\d_1\d_2}b_{\d_2}a_{\d_1},\tab \d_i\in\{\pm\},
}
or, in $q_*$ notation we have:
\Eqn{q_*a_-a_+-q_*\inv a_+a_-&=\frac{\mu^2}{q_*-q_*\inv},\\
q_*b_+b_--q_*\inv b_-b_+&=\frac{\mu^{-2}}{q_*-q_*\inv},\\
a_{\d_1}b_{\d_2}&=q_*^{2\d_1\d_2}b_{\d_2}a_{\d_1},
}
which is similar to the bosonic case considered in \cite{blz2}. Hence as in \cite{blz2}, the tensor product $\rho_+(\l\mu)\ox \rho_-(\l\mu\inv)$ decomposes as
\Eq{\rho_+(\l\mu)\ox \rho_-(\l\mu\inv)=\bo+_{m=0}^\oo \rho^{(m)},}
where
\Eq{\rho^{(m)}: |\rho_k^{(m)}\>=(a_++b_+)^k(a_+-\c b_+)^m |0\>_+\ox |0\>_-,}
for $k\in \Z_{\geq 0}$ and $\c\neq -q_*^{2n}, n\in \Z$ any constant. Note that $|\rho_k^{(m)}\>$ is even when $k+m$ is even.

Let $\mu=q_*^{\frac{s}{2}+\frac{1}{2}}$. Then the action of $\fb_+$ is given by
\Eqn{
\rho_+(\l\mu)\ox \rho_-(\l\mu\inv)(\cK_1)|\rho_k^{(m)}\>&=q^{-2(k+m)}|\rho_k^{(m)}\>=(-1)^{k+m} q_*^{-2(k+m)}|\rho_k^{(m)}\>,\\
\rho_+(\l\mu)\ox \rho_-(\l\mu\inv)(\cK_0)|\rho_k^{(m)}\>&=q^{2(k+m)}|\rho_k^{(m)}\>=(-1)^{k+m} q_*^{2(k+m)}|\rho_k^{(m)}\>,\\
\rho_+(\l\mu)\ox \rho_-(\l\mu\inv)(\cE_0)|\rho_k^{(m)}\>&=\l|\rho_{k+1}^{(m)}\>,\\
\rho_+(\l\mu)\ox \rho_-(\l\mu\inv)(\cE_1)|\rho_k^{(m)}\>&=\l[k]_{q_*}[s-k+1]_{q_*}|\rho_{k-1}^{(m)}\>+c_k^{(m)}|\rho_k^{(m-1)}\>,
}
where $c_k^{(m)}$ are constants not necessary in what follows. 

We observe that the representation of $\fb_+$ has a block diagonal form defined by $\rho^{(m)}$, which resembles the Verma module $\cW_s^\pm$ with a shift in the factors of $\cK_i$. Hence in the Grothendieck ring of representation of $\fb_+$ we obtain
\Eq{[\rho_+(\l\mu)\ox \rho_-(\l\mu\inv)]=\sum_{m=0}^\oo [U_{-s-2m}\ox \cW_s^{(-1)^m}(\l)]}
where $U_{p}$ is the 1-dimensional representation such that $\cE_1,\cE_0$ act trivially as 0, while $\cK_1,\cK_0$ acts as $q_*^p,q_*^{-p}$ respectively. Indeed, the action of $\cK_1$ on $U_{-s-2m}\ox \cW_s^{(-1)^m}(\l)$ is given by multiplication by
$$(q_*^{-s-2m})\cdot ((-1)^m(-1)^kq_*^{-2k+s})=(-1)^{k+m}q_*^{-2(k+m)}.$$

Note that $[U_0]=[W_0^+(\l)]=1$ in the Grothendieck ring.

Let us denote by
$$U^\pm_{-s-2m}:=\C\cdot |\pm\>$$
 the 1-dimensional representation with odd generator $|-\>$ or even generator $|+\>$.
(Here $U_p^+:=U_p$) We have
\Eq{U_m^{\e_1}\ox U_n^{\e_2}\simeq U_{m+n}^{\e_1\e_2}.}

Let us introduce the parity element $\s:=[U^-_{0}]$ in the Grothendieck ring. Then 
\Eq{U_0^- \ox \cW_s^\pm \simeq \cW_s^\mp,\tab U_0^-\ox U_p^\pm \simeq U_p^\mp.}
Hence
\Eq{\s[\cW_s^\pm]=[\cW_s^\mp],\tab \s[U_p^\pm]=[ U_p^\mp],\tab \s^2=1,}
and we can rewrite in the Grothendieck ring:
\Eqn{[\rho_+(q_*^{\frac{s}{2}+\frac{1}{2}}\l)][\rho_-(q_*^{-\frac{s}{2}-\frac{1}{2}}\l)]&=\sum_{m=0}^\oo [U_{-s-2m}\ox \cW_s^{(-1)^m}(\l)]\\
&=\sum_{m=0}^\oo \s^m[U_{-s-2m}][\cW_s^{+}(\l)]\\
&=[\cW_s^+(\l)] \sum_{m=0}^\oo \s^m [U_{-s-2m}]\\
&=[\cW_s^+(\l)]\cdot f_s
}
where 
\Eq{f_s:=\sum_{m=0}^\oo \s^m[U_{-s-2m}]=[U_{-s}]\sum_{m=0}^\oo \s^m[U_{-2}]^m=\frac{[U_{-s}]}{1-\s[U_{-2}]}.}

For simplicity, let us always fix the highest weight of the finite-dimensional module to be even and rewrite $W_s(\l):=W_s^+(\l)$.

Now from previous observation, $$[\cW_s^+(\l)]=[W_s(\l)]+[\cW_{-s-2}^{(-1)^{s+1}}(\l)]=[W_s(\l)]+\s^{s+1}[\cW_{-s-2}^+(\l)].$$
Letting $s\mapsto -s-2$, we have
$$[\rho_+(q_*^{-\frac{s}{2}-\frac{1}{2}}\l)][\rho_-(q_*^{\frac{s}{2}+\frac{1}{2}}\l)]=[\cW_{-s-2}^+(\l)]\cdot f_{-s-2}.$$
Hence we have
\Eqn{[W_s(\l)]&=[\cW_s^+(\l)]-\s^{s+1}[\cW_{-s-2}^+(\l)]\\
&=f_s\inv [\rho_+(q_*^{\frac{s}{2}+\frac{1}{2}}\l)][\rho_-(q_*^{-\frac{s}{2}-\frac{1}{2}}\l)]-f_{-s-2}\inv\s^{s+1}[\rho_+(q_*^{-\frac{s}{2}-\frac{1}{2}}\l)][\rho_-(q_*^{\frac{s}{2}+\frac{1}{2}}\l)].
}
In particular, letting $s=0$, we obtain the $q$-Wronskian identity:
\Eq{\label{qW} 
1=[W_0(\l)]=f_0\inv [\rho_+(q_*^{\frac{1}{2}}\l)][\rho_-(q_*^{-\frac{1}{2}}\l)]-f_{-2}\inv\s[\rho_+(q_*^{-\frac{1}{2}}\l)][\rho_-(q_*^{\frac{1}{2}}\l)].
}

On the other hand, let us consider the product of $[W_1(\l)]$ and $[\rho_+(\l)]$. Using \eqref{qW} with appropriate $\l$:
\Eqn{
[W_1(\l)][\rho_+(\l)]&=f_1\inv [\rho_+(q_*\l)][\rho_-(q_*^{-1}\l)][\rho_+(\l)]-f_{-3}\inv[\rho_+(q_*^{-1}\l)][\rho_-(q_*\l)][\rho_+(\l)]\\
&=f_1\inv [\rho_+(q_*\l)](f_0+f_{-2}\inv f_0\s[\rho_+(q_*^{-1}\l)][\rho_-(\l)]\\
&-f_{-3}\inv[\rho_+(q_*^{-1}\l)](f_0\inv f_{-2}\s[\rho_+(q_*\l)][\rho_-(\l)]-f_{-2}\s)\\
&=f_1\inv f_0 [\rho_+(q_*\l)]+f_1\inv f_{-2}\inv f_0\s[\rho_+(q_*\l)][\rho_+(q_*^{-1}\l)][\rho_-(\l)]\\
&-f_{-3}\inv f_0\inv  f_{-2}\s [\rho_+(q_*^{-1}\l)][\rho_+(q_*\l)][\rho_-(\l)]-f_{-3}\inv f_{-2}\s [\rho_+(q_*^{-1}\l)].
}
Now using
\Eqn{f_1\inv f_{-2}\inv f_0=f_{-3}\inv f_0\inv f_{-2} &= \frac{1-\s[U_{-2}]}{[U_1]}=f_{-1}\inv\\
f_1\inv f_0=[U_1],&\tab f_{-3}\inv f_{-2}=[U_{-1}],
}
we get the Baxter relation:
\Eq{[W_1(\l)][\rho_+(\l)]=[U_1][\rho_+(q_*\l)]-\s[U_{-1}][\rho_+(q_*\inv \l)].}

Similar relation holds for $[W_1(\l)]$ and $[\rho_-(\l)]$:
\Eq{[W_1(\l)][\rho_-(\l)]=[U_1][\rho_-(q_*\inv\l)]-\s[U_{-1}][\rho_-(q_* \l)].}

\section{Transfer matrices for SCFT}\label{sec:transfer}
The universal $R$-matrix for  $U_q(C^{(2)}(2))$ belongs to a completion of $\cU(\fb_+)\ox \cU(\fb_-)$.
In \cite{kz} the lower Borel subalgebra $\fb_-$ was represented by means of vertex operators (here we use some rescaling):
$$V_{\pm}(u)=\int d \theta :e^{\pm\Phi(u,\theta)}: = \mp {i}{\sqrt{2}}\xi(u):e^{\pm2\phi(u)}:,$$ 
where
\begin{eqnarray}
&&\Phi(u,\theta):=\phi(u)-\frac{i}{\sqrt{2}}\theta\xi(u)\\ 
&&\phi(u):=iQ+iPu+\sum_n\frac{a_{-n}}{n}e^{inu},\qquad
\xi(u):=i^{-1/2}\sum_n\xi_ne^{-inu},\nonumber\\
&&[Q,P]=\frac{ib^2}{2} ,\quad 
[a_n,a_m]=\frac{b^2}{2}n\delta_{n+m,0},\qquad
\{\xi_n,\xi_m\}=b^2\delta_{n+m,0}.\nonumber\\
&&:e^{\pm\phi(u)}:=
\exp\Big(\pm\sum_{n=1}^{\infty}\frac{2a_{-n}}{n}e^{inu}\Big)
\exp\Big(\pm 2i(Q+Pu)\Big)\exp\Big(\mp\sum_{n=1}^{\infty}\frac{2a_{n}}{n}e^{-inu}
\Big)\nonumber.
\end{eqnarray} 
These are the vertex operators acting in the Fock space and according to their commutation relations, the substitution

\Eqn{H_{\a_1}\to \frac{2P}{b^2},&\tab \cE_{-\a_1}= \int_0^{2\pi} V_-(u)du\\
H_{\a_0}\to -\frac{2P}{b^2},&\tab \cE_{-\a_0}= \int_0^{2\pi} V_+(u)du}
gives rise to a representation of the lower Borel subalgebra $\fb_-$ with  $q=e^{\pi i b^2}$. 

The $R$-matrix with $\fb_-$ represented as above and $\fb_+$ as in $W_s(\l)$ has the form
\Eq{\bL_s(\l)=e^{\pi iP\cH} Pexp^{(q)} \int_0^{2\pi} (\l V_-(u)\cE+\l V_+(u)\cF)du}

The letter $q$ over the path-ordered exponential ($Pexp$) means certain regularization procedure, which preserves the property of $Pexp$ (see \cite{kz} for more details).

Similarly, one can consider operators $\bL_\pm(\l)$, where the upper Borel algebra $\fb_+$ is represented via $\rho_\pm(\l)$:
\Eq{\bL_{\pm}(\l)=e^{\pm \pi iP\cH} Pexp^{(q)}\int_0^{2\pi}(\l V_-(u)\a_\pm + \l V_+(u)\a_\mp)du}
Then define
\begin{eqnarray}
&& \bT_s(\l):=sTr(e^{\pi iP\cH} \bL_s(\l)), \quad \bT^+_s(\l):=sTr(e^{\pi iP\cH} \bL_s(\l))\nonumber\\
&& \til{\bQ}_{\pm}(\l):=sTr (e^{\pm\pi iP\cH}\bL_\pm(\l)),
\end{eqnarray}
where we consider the highest weight vector in $W_s(\l),\rho_{\pm}(\l)$ to be even, and we take the supertrace of the representation of the second tensor factor. (We ignore the convergence of the trace here, treating it as formal series in $\l$.)

Then from the previous decomposition and the properties of the supertrace
\Eqn{
&\til{\bQ}_{+}(q_*^{\frac{s}{2}+\frac{1}{2}}\l)\til{\bQ}_-(q_*^{-\frac{s}{2}-\frac{1}{2}}\l)\\
&=sTr (e^{\pi i P \cH}\bL_+(q_*^{\frac{s}{2}+\frac{1}{2}}\l))sTr (e^{-\pi i P \cH}\bL_-(q_*^{-\frac{s}{2}-\frac{1}{2}}\l))\\
&=sTr_{\rho_+(q_*^{\frac{s}{2}+\frac{1}{2}}\l)}(e^{\pi iPH}\bR)sTr_{\rho_-(q_*^{-\frac{s}{2}-\frac{1}{2}}\l)}(e^{\pi iPH}\bR)\\
&=\sum_{m=0}^\oo sTr_{\cW_s^+(\l)}(e^{\pi iPH}\bR))sTr_{U_{-s-2m}^{(-1)^m}}(e^{\pi iPH}\bR)\\
&=\sum_{m=0}^\oo sTr(e^{\pi iP\cH}\bL_+(\l))sTr_{U_{-s-2m}^{(-1)^m}}(e^{\pi iPH}\bR)\\
&=\sum_{m=0}^\oo\bT_s^+(\l)sTr_{U_{-s-2m}^{(-1)^m}}(e^{2\pi iP\cH})\\
&=\sum_{m=0}^\oo(\bT_s(\l)+(-1)^{s+1}\bT_{-s-2}^+(\l))sTr_{U_{-s-2m}^{(-1)^m}}(e^{2\pi iP\cH})\\
&=(\bT_s(\l)+(-1)^{s+1}\bT_{-s-2}^+(\l))\sum_{m=0}^\oo(-1)^m e^{2\pi iP_*(-s-2m)}\\
&=\frac{e^{-2\pi iP_*(s-1)}}{2\cos(2\pi P_*)}(\bT_s(\l)+(-1)^{s+1}\bT_{-s-2}^+(\l)),
}
where $P_*=\frac{b_*^2}{b^2}P$.
Define the rescaled operator \Eq{\bQ_{\pm}(\l):=2\cos(2\pi P_*)e^{2\pi iP_*}(\l)^{\pm\frac{2P_*}{b_*^2}} \til{\bQ}_{\pm}(\l).}
Then
$$\bQ_+(q_*^{\frac{s}{2}+\frac{1}{2}}\l)\bQ_-(q_*^{-\frac{s}{2}-\frac{1}{2}}\l)=2\cos(2\pi P_*)(\bT_s(\l)+(-1)^{s+1}\bT_{-s-2}^+(\l)).$$
Together with the other relation by substituting $s\to -s-2$:
$$\bQ_+(q_*^{-\frac{s}{2}-\frac{1}{2}}\l)\bQ_-(q_*^{\frac{s}{2}+\frac{1}{2}}\l)=2\cos(2\pi P_*)\bT_{-s-2}^+(\l),$$
we have
\Eq{2\cos(2\pi P_*)\bT_s(\l)=\bQ_+(q_*^{\frac{s}{2}+\frac{1}{2}}\l)\bQ_-(q_*^{-\frac{s}{2}-\frac{1}{2}}\l)+(-1)^s\bQ_+(q_*^{-\frac{s}{2}-\frac{1}{2}}\l)\bQ_-(q_*^{\frac{s}{2}+\frac{1}{2}}\l).}
In particular, we obtain the quantum super-Wronskian relation:
\Eq{2\cos(2\pi P_*)=\bQ_+(q_*^{\frac{1}{2}}\l)\bQ_-(q_*^{-\frac{1}{2}}\l)+\bQ_+(q_*^{-\frac{1}{2}}\l)\bQ_-(q_*^{\frac{1}{2}}\l).}

The Baxter T-Q relations for $\bQ$-operator follows from previous section:
\Eq{\bT_1(\l)\cdot \bQ_{\pm}(\l)=\pm\bQ_{\pm}(q_*\l)\mp\bQ_{\pm}(q_*\inv\l).}
The fusion relation, which follows from the quantum super-Wronskian relation 
is:
\Eq{\bT_s(q_*^{\frac{1}{2}}\l)\bT_s(q_*^{-\frac{1}{2}}\l)=\bT_{s+1}(\l)\bT_{s-1}(\l)+(-1)^s.}
This relation is similar to the one considered in \cite{kz}, but now all the transfer matrices correspond to the representations of  $U_q(C^{(2)}_q(2))$. 
In particular, 
\Eq{\bT_2(\l)=\bT_1(q_*^{\frac{1}{2}}\l)\bT_1(q_*^{-\frac{1}{2}}\l)+1.}
Therefore
$$\bT_2(q_*^{\frac{1}{2}}\l)=\bT_1(q_*\l)\bT_1(\l)+1,$$ so that the Baxter relation for $\bT_2$ is as follows.
\Eqn{\bT_2(q_*^{\frac{1}{2}}\l)\bQ_\pm(\l)&=\bQ_\pm(\l)+\bT_1(q_*\l)(\pm\bQ_\pm(q_*\l)\mp\bQ_\pm(q_*\inv\l))\\
&=\bQ_\pm(\l)+\bQ_\pm(q_*^2\l)-\bQ_\pm(\l)\mp \bT_1(q_*\l)\bQ_\pm(q_*\inv \l)\\
&=\bQ_\pm(q_*^2\l)\mp \bT_1(q_*\l)\bQ_\pm(q_*\inv \l)
}
Moreover, one can write down the expression for each $\bT_s$ in terms of either one of  $\mathbf{Q}_{\pm}(\lambda)$ using the quantum super-Wronskian relation:
\begin{eqnarray}
\mathbf{T}_s(\lambda)=\mathbf{Q}_{\pm}(q_*^{\frac{s}{2}+\frac{1}{2}}\lambda)\mathbf{Q}_{\pm}(q_*^{-\frac{s}{2}-\frac{1}{2}}\lambda)
\sum^{s/2}_{k=-s/2}
\frac{(-1)^{(k\pm \frac{s}{2})}}{\mathbf{Q}_{\pm}(q_*^{k+\frac{1}{2}}\lambda)\mathbf{Q}_{\pm}(q_*^{k-\frac{1}{2}}\lambda)}
\end{eqnarray}
The $\bT_2$-transfer matrix has a classical limit of the trace of monodromy matrix for super-KdV equation. The asymptotic expansion of it should produce both local and nonlocal integrals of motion for superconformal field theory (SCFT). We suppose that operators $\mathbf{Q}_{\pm}(\lambda)$ possess nice analytic properties like it was in the $A^{(1)}_1$ case \cite{blz1}.

\begin{appendices}
\section{Appendix}
Let us introduce the $q$-numbers:
$$[n]_q=\frac{q^n-q^{-n}}{q-q\inv}$$
such that
\Eqn{[n]_{q_*}&=\frac{q_*^n-q_*^{-n}}{q_*-q_*\inv}\\
&=i^{n-1} \frac{q^{-n}-(-1)^{n}q^{n}}{q+q\inv}\\
&=i^{n-1} \{n\}_q
}
with the usual notation in superalgebra $$\{n\}_q := \frac{q^{-n}-(-1)^{n}q^{n}}{q+q\inv}.$$

\subsection{$R$-matrix for $U_q(\osp(2|1))$}

Let us prove Theorem \ref{R_osp} that the universal $R$ matrix is given by
\Eq{\bR=Q \cR,}
where $Q=C q_*^{\frac{H\ox H}{2}}$ with
$C=\frac{1}{2}(1\ox 1+i\xi\eta\ox 1+1\ox i\xi\eta +\xi\eta\ox \xi\eta)$
such that $$C\cdot |(-1)^{\e_1}\>\ox |(-1)^{\e_2}\> = (-1)^{\e_1\e_2}|(-1)^{\e_1}\>\ox |(-1)^{\e_2}\>,\tab \e_i\in\{0,1\},$$
and
\Eqn{\cR&=\exp_{q_*^{-2}}(i(q_*\inv-q_*) \cE\ox \cF)\\
&=\exp_{-q^{-2}}(-(q+q\inv) \cE\ox \cF)\\
&=\sum a_n \cE^n \ox \cF^n.
}
Note that using  
$$\lceil n \rceil_{q_*^{-2}}=(-q)^{1-n}\{n\}_q,$$ we have
$$a_n=(-1)^nq^{\frac{1}{2}n(n-1)}\frac{(q+q\inv)^n}{\{n\}_q!}.$$
By definition $a_n$ satisfies 
\Eq{\frac{a_n}{a_{n-1}}=-q^{n}\frac{(1+q^{-2})}{\{n\}_q}.}
The properties of an $R$-matrix states that
\Eq{\D^{op}(X)\bR = \bR \D(X),\tab X\in \cU_q{\osp(2|1)}}
i.e. on the generators we have
\Eq{(\cK\ox \cE+\cE\ox 1) \cR=\cR(1\ox \cE+\cE\ox \cK)}
\Eq{(1\ox \cF+\cF\ox \cK\inv) \cR=\cR(\cK\inv\ox \cF+\cF\ox 1)}
\Eq{(\cK\ox \cK) \cR=\cR(\cK\ox \cK)}
In order to prove that R satisfies the properties of the $R$-matrix, one check that
\Eq{
(\cK\ox \cE)Q&=Q(1\ox \cE)\\
(\cE\ox 1)Q&=Q(\cE\ox \cK\inv)
}
which follows easily from the commutation relations of the Clifford algebra, and 
\Eq{(1\ox \cE+\cE\ox \cK\inv) \cR=\cR(1\ox \cE+\cE\ox \cK)}
The calculation for $\cF$ is similar, while the relation for $\cK$ is trivial since it commutes with every term.
Using
\Eq{\cE\cF^n- (-1)^n \cF^n \cE =  \frac{q^{n}\{n\}_q}{1+q^2}\cK\cF^{n-1}+ \frac{(-1)^nq^{-n}\{n\}_q}{1+q^{-2}} K\inv\cF^{n-1},}
we have
\Eqn{(\cE\ox \cK\inv)(\cE^n\ox \cF^n)-(\cE^n\ox \cF^n)(\cE\ox \cK)&=\cE^{n+1}\ox \cK\inv \cF^n - (-1)^nq^{2n} \cE^{n+1}\ox \cK\cF^n\\
(1\ox \cE)(\cE^n\ox \cF^n)-(\cE^n\ox \cF^n)(1\ox \cE)&=(-1)^n\cE^n\ox \cE\cF^n -\cE^n\ox \cF^n\cE\\
&= (-1)^n\frac{q^{n}\{n\}_q}{1+q^{2}} \cE^n\ox \cK\cF^{n-1}+\frac{q^{-n}\{n\}_q} {1+q^{-2}} \cE^n\ox K\inv\cF^{n-1}
}
Hence adding up both sides, we need $a_0=1$ and
\Eqn{a_{n-1}+\frac{q^{-n}\{n\}_q}{1+q^{-2}}a_n &=0\\
(-1)^{n-1}q^{2(n-1)}a_{n-1}- (-1)^n\frac{q^{n}\{n\}_q}{1+q^{2}} a_n &=0}
both of which is equivalent to
$$\frac{a_n}{a_{n-1}}=-q^n\frac{(1+q^{-2})}{\{n\}_q}$$
as required.

By writing formally $$K=q^{H'}=i\xi\eta q_*^H,$$ the following proposition shows that up to a constant, the Cartan part of the universal $R$-matrix using the Clifford generators coincides with the usual expression.
\begin{Prop} On the space $W_{s_1}^{\pm_1} \ox W_{s_2}^{\pm_2}$, we have the action
\Eq{q^{\frac{H'\ox H'}{2}} = (-1)^{-l_1l_2}\til{q} Cq_*^{\frac{H\ox H}{2}},}
where $C$ is the Clifford part $\frac{1}{2}(1\ox 1+i\xi\eta\ox 1+1\ox i\xi\eta +\xi\eta\ox \xi\eta)$ and
$H'$ reproduce the action of $K$ on $W_s^\pm$: 
$$H'=\case{H-{l}\frac{\pi i}{\ln q}& +\\H-(l+1)\frac{\pi i}{\ln q}& -}$$
with $s=2l$.
\end{Prop}
\begin{proof}
For simplicity, consider the action on the basis $e_{m}^{l_1}\ox e_n^{l_2}\in W_{s_1}^{+} \ox W_{s_2}^{+}$. The action on other parity is similar.
Then we have
\Eqn{q^{\frac{H'\ox H'}{2}}&=q^{\frac{H\ox H}{2}}(i^{-l_2H} \ox 1)(1\ox i^{-l_1 H}) \til{q}\\
&=q^{2mn}(-1)^{-l_2m-l_1n}\til{q}}
while
\Eqn{Cq_*^{\frac{H\ox H}{2}}&=(-1)^{(l_1-m)(l_2-n)}q_*^{2mn}\\
&=(-1)^{-mn}q^{2mn}(-1)^{(l_1-m)(l_2-n)}\\
&=(-1)^{l_1l_2-l_2m-l_1n}q^{2mn}}
\end{proof}

\subsection{Universal $R$ matrix for $U_q(C^{(2)}(2))$}
Recall from \eqref{rescale} that we have rescaled our generator $\cF_i$ from the usual definition by $c=\frac{q+q\inv}{q-q\inv}$. Hence modifying the constants from \cite{kt} accordingly, the universal $R$ matrix in general is of the form
\Eq{\bR=Q \cR_{>0}\cR_0\cR_{<0},}
where
\Eq{Q=q^{\frac{H_1\ox H_1}{2}+H_\d\ox H_d+H_d\ox H_\d},}
with $H_\d=H_0+H_1$ and $H_d$ the extended generators such that $$[H_d,\cE_0]=\cE_0, \tab[H_d, \cE_1]=0,$$ and
\Eqn{\cR_{>0}&=\prod_{n\geq 0}\exp_{-q^{-2}}((-1)^{n+1}(q\inv+q)\cE_{\a+n\d}\ox \cF_{\a+n\d}),\\
\cR_{<0}&=\prod_{n\geq 0}\exp_{-q^{-2}}((-1)^{n+1}(q\inv+q)\cE_{\d-\a+n\d}\ox \cF_{\d-\a+n\d}),\\
\cR_0&=\exp\left(\sum_{n>0} \frac{n(q+q^{-1})^2}{q^{2n}-q^{-2n}} \cE_{n\d}\ox \cF_{n\d}\right),}
where the imaginary generators $\cE_{n\d\pm\a},\cF_{n\d\pm\a}$ are defined below.
\begin{Prop} The Cartan term can be replaced using the Clifford part:
\Eq{Q=C q_*^{\frac{H_1\ox H_1}{2}+H_\d\ox H_d+H_d\ox H_\d}.}
\end{Prop}
\begin{proof} We just need to check that the following same commutation holds:
\Eqn{(\cE_1\ox 1)Q=(\cE_1\ox \cK_1\inv),&\tab (\cE_0\ox 1)Q=(\cE_0\ox \cK_0\inv)\\
(\cF_1\ox \cK_1\inv)Q=(\cF_1\ox 1),&\tab (\cF_0\ox \cK_0\inv)Q=(\cF_0\ox 1)\\
(\cK_1\ox \cE_1)Q=(1\ox \cE_1),&\tab (\cK_0\ox \cE_0)Q=(1\ox \cE_0)\\
(1\ox \cF_1)Q=(\cK_1\ox \cF_1),&\tab (1\ox \cF_0)Q=(\cK_0\ox \cF_0)}
Then it follows that the Clifford part $C$ commute correctly with the odd elements because $\cE_i=E_i\xi,\cF_i=F_i\eta$ and $\cK_i=K_ii\xi\eta$ as before, and the even part follows from the relation of $U_{q^*}(A_1^{(1)})$.
\end{proof}

Let us define the following notations for the generators:
\Eqn{
\cE_1:=\cE_\a,&\tab \cE_0:=\cE_{\d-\a}\\
\cF_1:=\cF_\a, &\tab \cF_0:=\cF_{\d-\a}\\
\cK_1:=\cK_\a, &\tab \cK_0:=\cK_{\d-\a}\\
\cK_\d:=\cK_\a \cK_{\d-\a}.&
}

Then using
\Eq{[e_\b,e_{\b'}]_q:=e_\b e_{\b'}-(-1)^{\h(\b)\h(\b')}q^{(\b,\b')}e_{\b'}e_\b,}
where $\h(\b)$ is the parity of $e_\b$, we define
\Eqn{\cE_\d:=[\cE_\a,\cE_{\d-\a}]_q&=\cE_1\cE_0+q^{-2}\cE_0\cE_1,\\
\cF_\d:=[\cF_{\d-\a},\cF_{\a}]_{q\inv}&=\cF_0\cF_1+q^{2}\cF_1\cF_0.}
Both $\cE_\d, \cF_\d$ are even.

Next we define
\Eqn{\cE_{n\d+\a}:=\frac{1}{q-q\inv}[\cE_{(n-1)\d+\a},\cE_\d],\\
\cF_{n\d+\a}:=\frac{1}{q-q\inv}[\cF_\d,\cF_{(n-1)\d+\a}],\\
\cE_{(n+1)\d-\a}:=\frac{1}{q-q\inv}[\cE_\d,\cE_{n\d-\a}],\\
\cF_{(n+1)\d-\a}:=\frac{1}{q-q\inv}[\cF_{n\d-\a},\cF_\d].}
These are all odd.

The pure imaginary roots are harder to define. First we define
$$\cE_{n\d}':=[\cE_\a,\cE_{n\d-\a}]_q=\cE_\a \cE_{n\d-\a}+q^{-2} \cE_{n\d-\a}\cE_\a,$$
$$\cF_{n\d}':=[\cF_{n\d-\a},\cF_\a]_{q\inv}=\cF_{n\d-\a}\cF_\a +q^{2}\cF_\a \cF_{n\d-\a}.$$
Note that $\cE_{\d}'=\cE_\d, \cF_{\d}'=\cF_{\d}$.
Then the pure imaginary root vectors are defined recursively by
$$\cE_{n\d}=\sum_{p_1+2p_2+...+np_n=n}\frac{(q-q\inv)^{\sum p_i-1}(\sum p_i-1)!}{p_1!...p_n!}(\cE_\d')^{p_1}...(\cE_{n\d}')^{p_n},$$
$$\cF_{n\d}=\sum_{p_1+2p_2+...+np_n=n}\frac{(q\inv-q)^{\sum p_i-1}(\sum p_i-1)!}{p_1!...p_n!}(\cF_{n\d}')^{p_n}...(\cF_{\d}')^{p_1}.$$
More explicitly, by using generating functions:
\Eqn{\bE'(u)&:=-(q+q\inv)\sum_{n\geq 1}\cE_{n\d}' u^{-n},\\
\bE(u)&:=-(q+q\inv)\sum_{n\geq 1}\cE_{n\d} u^{-n},}
\Eqn{\bF'(u)&:=(q+q\inv)\sum_{n\geq 1}\cF_{n\d}' u^{-n},\\
\bF(u)&:=(q+q\inv)\sum_{n\geq 1}\cF_{n\d} u^{-n},}
we have
$$\bE'(u)=-1+\exp \bE(u),\tab \bE(u)=\ln(1+\bE'(u))$$
and similarly for $\bF(u)$.

\begin{Prop} We have the following action of the non-simple generators on $W_s^\pm(\l)$:

\Eqn{
\cE_\d \cdot e_{m,\pm}^l&=\l^2q_*^{-m-1}\left(q_*^l[l+m+1]_{q_*}-q_*^{-l}[l-m+1]_{q_*}\right) e_{m,\pm}^l\\
\cF_\d \cdot e_{m,\pm}^l&=\l^{-2}q_*^{m+1}\left(q_*^{l}[l-m+1]_{q_*}-q_*^{-l}[l+m+1]_{q_*}\right) e_{m,\pm}^l\\
\cE_{n\d+\a}\cdot e_{m,\pm}^l&=i^n\l^{2n+1}q_*^{-2n(m+1)}[l-m]_{q_*}e_{m+1,\mp}^l\\
\cE_{(n+1)\d-\a}\cdot e_{m,\pm}^l&=i^n\l^{2n+1}q_*^{-2nm}[l+m]_{q_*}e_{m-1,\mp}^l\\
\cF_{n\d+\a}\cdot e_{m,\pm}^l&=\pm i^{n-1}\l^{-2n-1}q_*^{2nm}[l+m]_{q_*}e_{m-1,\mp}^l\\
\cF_{(n+1)\d-\a}\cdot e_{m,\pm}^l&=\pm i^{n-1}\l^{-2n-1}q_*^{2n(m+1)}[l-m]_{q_*}e_{m+1,\mp}^l\\
\cE_{n\d}'\cdot e_{m,\pm}^l&=i^{n-1}\l^{2n}q_*^{-2(n-1)m}\left([l+m]_{q_*}[l-m+1]_{q_*}-q_*^{-2n}[l-m]_{q_*}[l+m+1]_{q_*}\right)e_{m,\pm}^l\\
\cF_{n\d}'\cdot e_{m,\pm}^l&=i^{n-1}\l^{-2n}q_*^{2(n-1)m}\left([l+m]_{q_*}[l-m+1]_{q_*}-q_*^{2n}[l-m]_{q_*}[l+m+1]_{q_*}\right)e_{m,\pm}^l.
}
By the generating functions, we get
\Eqn{
\cE_{n\d}\cdot e_{m,\pm}^l&=i^{n-1}\frac{\l^{2n}}{n}N(l,m,n,q_*)e_{m,\pm}^l,\\
\cF_{n\d}\cdot e_{m,\pm}^l&=i^{n-1}\frac{\l^{-2n}}{n}N(l,m,n,q_*\inv)e_{m,\pm}^l,
}
where 
\Eqn{N(l,m,n,q)&:=q^{-n(m+1)}(q^{n(l+1)}[n(l+m)]_{q}-q^{-n(l+1)}[n(l-m)]_{q})\\
&=\frac{q^{2nl}+q^{-2n(l+1)}-q^{-2nm}-q^{-2n(m+1)}}{q-q\inv}.}
\end{Prop}

\begin{Thm}\label{RC2full} We have the following expression for $R$:
$$\bR=Q\cR_{>0}\cR_0\cR_{<0},$$
where the matrix coefficients of each component are given below expressed only in terms of $q_*$:
\begin{itemize}
\item The matrix coefficients of $\cR_{>0}$ is given by:
$$\< e_{m_1',\e_1'}^{l_1}\ox e_{m_2',\e_2'}^{l_1}|\cR_{>0}| e_{m_1,\e_1}^{l_1}\ox e_{m_2,\e_2}^{l_2}\>=0$$
if $m_1'-m_1\neq m_2-m_2'$ or $m_1'-m_1=m_2-m_2'<0$.

Otherwise let $n=m_1'-m_1$, we have
\Eqn{&\< e_{m_1',\e_1'}^{l_1}\ox e_{m_2',\e_2'}^{l_2}|\cR_{>0}| e_{m_1,\e_1}^{l_1}\ox e_{m_2,\e_2}^{l_2}\>\\
&=\frac{(-1)^{n(\e_1+\e_2-1)}(q_*-q_*\inv)^n(\l_1\l_2)^n}{ \prod_{k=1}^n (\l_2^2-{q_*}^{2m_2-2m_1-2k}\l_1^2)}\frac{[ l_1-m_1]_{q_*}!}{\lfloor n \rfloor_{q_*}![ l_1-m_1-n]_{q_*}!}\frac{[ l_2+m_2]_{q_*}!}{[ l_2+m_2-n]_{q_*}!}}
where $\lfloor n \rfloor_{q_*}=\frac{1-q_*^{-2n}}{1-q_*^{-2}}$.

\item Similarly, the matrix coefficients of $\cR_{<0}$ is given by
$$\< e_{m_1',\e_1'}^{l_1}\ox e_{m_2',\e_2'}^{l_2}|\cR_{<0}| e_{m_1,\e_1}^{l_1}\ox e_{m_2,\e_2}^{l_2}\>=0$$
if $m_1'-m_1\neq m_2-m_2'$ or $m_1'-m_1=m_2-m_2'>0$.

Otherwise let $n=m_1-m_1'$, we have
\Eqn{
&\< e_{m_1',\e_1'}^{l_1}\ox e_{m_2',\e_2'}^{l_2}|\cR_{<0}| e_{m_1,\e_1}^{l_1}\ox e_{m_2,\e_2}^{l_2}\>=0\\
&=\frac{(-1)^{n(\e_1+\e_2-1)}(q_*-q_*\inv)^n(\l_1\l_2)^n}{ \prod_{k=1}^n (\l_2^2-{q_*}^{2m_2-2m_1+2(k+n-1)}\l_1^2)}\frac{[ l_1+m_1]_{q_*}!}{\lfloor n \rfloor_{q_*}![ l_1+m_1-n]_{q_*}!}\frac{[ l_2-m_2]_{q_*}!}{[ l_2-m_2-n]_{q_*}!}
}

\item The matrix coefficients of $\cR_0$ is given by
$$\cR_{0}(e_{m_1,\e_1}^{l_1}\ox e_{m_2,\e_2}^{l_2})=f_q\cdot\prod_{k=1}^{l_1+m_1}\frac{\l_2^2-\l_1^2 q_*^{2l_1+2l_2-2k+2}}{\l_2^2-\l_1^2q_*^{2m_2-2m_1+2k}}\prod_{k=1}^{l_2+m_2}\frac{\l_2^2-\l_1^2 q_*^{2m_2-2m_1-2k}}{\l_2^2-\l_1^2q_*^{-2l_1-2l_2+2k-2}}e_{m_1,\e_1}^{l_1}\ox e_{m_2,\e_2}^{l_2},$$

where
\Eqn{f_q(l_1,\l_1,l_2,\l_2)&=\exp\left(\sum_{n>0}\frac{1}{n}\left(\frac{\l_1}{\l_2}\right)^{2n}\frac{(q_*^{2l_1n}-q_*^{-2l_1n})(q_*^{2l_2n}-q_*^{-2l_2n})}{q_*^{2n}-q_*^{-2n}}\right)\\
&=\exp\left(\sum_{n>0}\frac{1}{n}\left(\frac{\l_1}{\l_2}\right)^{2n}[2l_1]_{q_*^n}[2l_2]_{q_*^n}\frac{q_*^{n}-q_*^{-n}}{q_*^{n}+q_*^{-n}}\right).}

\item Finally, the action of $Q$ is given by
$$Q(e_{m_1,\e_1}^{l_1}\ox e_{m_2,\e_2}^{l_2})=(-1)^{\e_1\e_2}q_*^{2m_1 m_2}e_{m_1,\e_1}^{l_1}\ox e_{m_2,\e_2}^{l_2}.$$
\end{itemize}
\end{Thm}

\begin{Ex}
When $l_1=l_2=\frac{1}{2}$, we have
$$\bR_{\frac{1}{2},\frac{1}{2}}(\l_1,\l_2)=q_*^{\frac{1}{2}}f_{q_*}\veca{1&0&0&0\\0&\frac{\l_1^2-\l_2^2}{\l_1^2q_*^{-1}-\l_2^2q_*}&\frac{\l_1\l_2(q_*\inv-q_*)}{\l_1^2q_*^{-1}-\l_2^2q_*}&0\\0&\frac{\l_1\l_2(q_*\inv-q_*)}{\l_1^2q_*^{-1}-\l_2^2q_*}&\frac{\l_1^2-\l_2^2}{\l_1^2q_*^{-1}-\l_2^2q_*}&0\\0&0&0&-1}$$
where $$f_{q_*}(\l_1,\l_2):=\exp\left(\sum_{n>0}\frac{1}{n}\left(\frac{\l_1}{\l_2}\right)^{2n}\frac{q_*^n-q_*^{-n}}{q_*^n+q_*^{-n}}\right).$$
Note that up to a constant we recover our previous formula \eqref{RC2}.
\end{Ex}
\begin{Ex}

Using Theorem \ref{RC2full}, we found for example the universal $R$ matrix acting on $W_2^+\ox W_2^+$ is given by

$$\bR_{1,1}(\l_1,\l_2)=\frac{q_*^2f_q}{a}\veca{
a&.&.&.&.&.&.&.&.\\
.&b&.&d&.&.&.&.&.\\
.&.&c&.&f&.&g&.&.\\
.&d&.&b&.&.&.&.&.\\
.&.&-h&.&-e&.&-h&.&.\\
.&.&.&.&.&b&.&d&.\\
.&.&g&.&f&.&c&.&.\\
.&.&.&.&.&d&.&b&.\\
.&.&.&.&.&.&.&.&a\\}
$$
where $\l_1=e^{x_1},\l_2=e^{x_2}$,
\Eqn{
a&=4\sinh(x_1-x_2-\ln q_*)\sinh(x_1-x_2-2\ln q_*)\\
b&=4\sinh(x_1-x_2)\sinh(x_1-x_2-\ln q_*)\\
c&=4\sinh(x_1-x_2)\sinh(x_1-x_2+\ln q_*)\\
d&=-4\sinh(x_1-x_2-\ln q_*)\sinh(2\ln q_*)\\
e&=2\cosh(2x_1-2x_2-\ln q_*)-4\cosh(\ln q_*)+2\cosh(3\ln q_*)\\
f&=4q_*^{-1}\sinh(x_1-x_2)\sinh(\ln q_*)\\
g&=4\sinh(\ln q_*)\sinh(2\ln q_*)\\
h&=8q_*\sinh(x_1-x_2)\cosh(\ln q_*)\sinh(2\ln q_*)
}
and all other entries are zero.
\end{Ex}
\end{appendices}

\section*{Acknowledgments}
The first author is supported by World Premier International Research Center Initiative (WPI Initiative), MEXT, Japan.

\end{document}